\documentclass[11pt]{article}

\usepackage[margin=1in]{geometry}
\usepackage{amsmath,amssymb,amsthm,mathtools}
\usepackage{booktabs,tabularx,array}
\usepackage{enumitem}
\usepackage{microtype}
\usepackage{xcolor}
\usepackage[colorlinks=true,linkcolor=blue!55!black,citecolor=blue!55!black,urlcolor=blue!55!black]{hyperref}

\allowdisplaybreaks
\setlist[itemize]{leftmargin=*,topsep=3pt,itemsep=2pt}

\newtheorem{theorem}{Theorem}[section]
\newtheorem{lemma}[theorem]{Lemma}
\newtheorem{proposition}[theorem]{Proposition}

\theoremstyle{definition}
\newtheorem{definition}[theorem]{Definition}
\newtheorem{remark}[theorem]{Remark}

\newcommand{\R}{\mathbb{R}}
\newcommand{\E}{\mathbb{E}}
\newcommand{\Prob}{\mathbb{P}}
\newcommand{\eps}{\varepsilon}
\newcommand{\ind}{\mathbf{1}}
\newcommand{\KL}{D_{\mathrm{KL}}}
\newcommand{\op}{\mathrm{op}}
\newcommand{\sech}{\operatorname{sech}}

\newcommand{\calF}{\mathcal{F}}
\newcommand{\calZ}{\mathcal{Z}}
\newcommand{\calP}{\mathcal{P}}
\newcommand{\calG}{\mathcal{G}}
\newcommand{\od}{\odot}

\title{A Tight Lower Bound for Smooth Nonconvex Stochastic Optimization
with Bounded Gradient Noise}
\author{Jikai Jin}
\date{\today}

\begin{document}
\maketitle

\begin{abstract}
We prove a sharp lower bound for smooth nonconvex stochastic optimization
with uniformly bounded gradient noise.  In the \(K=1\) fresh-sample model,
every randomized adaptive algorithm requires
\[
  \Omega\!\left(
    \frac{\Delta L}{\eps^2}
    +
    \frac{\Delta L\sigma^2}{\eps^4}
  \right)
\]
queries to find a point with expected gradient norm at most \(\eps\).
This matches the standard upper bound and, to the best of our
knowledge, resolves the question raised by~\cite{ArjevaniEtAl2023} of
whether almost-surely bounded oracle error permits a better rate
than bounded variance.  

The proof was independently generated with GPT-5.6 Sol in Codex's Ultra mode
during a two-hour session.
The human author supplied the prompt and was responsible only for
checking the proof and revising and polishing the manuscript.
\end{abstract}

\section{Introduction}

For smooth nonconvex objectives, the standard first-order target is an
$\eps$-stationary point.  In the stochastic setting, stochastic gradient
descent achieves the familiar complexity
\[
  O\!\left(\frac{\Delta L}{\eps^2}
  +\frac{\Delta L\sigma^2}{\eps^4}\right)
\]
under an unbiased bounded-variance oracle; see, for example,
Ghadimi and Lan~\cite{GhadimiLan2013}.  Arjevani et
al.~\cite{ArjevaniEtAl2023} proved that this rate is minimax optimal in the
bounded-variance class.  Their rare-revelation construction, however, does
not satisfy a uniform almost-sure bound on the oracle error.  In their
discussion they explicitly identified the extension to almost-surely bounded
error as a nontrivial open direction.

This manuscript closes that direction for the $K=1$ fresh-sample oracle
model defined in Section~\ref{sec:model}.  The main difficulty is that a rare unbiased reveal in prior constructions normally has size
inversely proportional to its reveal probability and therefore violates the
bounded-noise constraint. Instead, we represent the next chain coordinate
by a random sign vector and distribute its gradient contribution across many
bounded signs with small biases.  Every response has fixed Euclidean
magnitude, but its Kullback--Leibler information about the unknown sign
vector remains small.  An additional $\log\cosh$ term prevents an algorithm
from escaping to large-norm points where the encoded chain might otherwise
become spuriously stationary.

\section{Related work and status of the question}
\label{sec:related}

\paragraph{Deterministic zero-chains.}
The scalar hard function that we use is a rescaled version of the smooth zero-chain used
by Carmon, Duchi, Hinder, and Sidford~\cite{CarmonEtAl2020}.  Their work
establishes the sharp deterministic $\Theta(\Delta L\eps^{-2})$ complexity
for smooth high-dimensional nonconvex optimization.  Random embeddings make
such sequential constructions hard for general, rather than merely
zero-respecting, algorithms.

\paragraph{Bounded variance and mean-squared smoothness.}
Arjevani et al.~\cite{ArjevaniEtAl2023} obtain the sharp
$\Omega(\Delta L\sigma^2\eps^{-4})$ lower bound under bounded second moment,
and an $\eps^{-3}$ lower bound under a mean-squared-smooth oracle model.
Variance-reduced methods such as SPIDER~\cite{FangEtAl2018} attain the latter
dependence when common randomness can be used at two query points.  These
results do not imply the theorem below: bounded variance permits arbitrarily
large individual errors, while uniform almost-sure boundedness is a strictly
smaller oracle class. Arjevani et al.~\cite[Section~6]{ArjevaniEtAl2023} state that extending their
lower bounds to almost-surely bounded error is nontrivial and leave the role
of that assumption for future work.

\paragraph{Other related results.}
Recent heavy-tail work, including Fradin et al.~\cite{FradinEtAl2026}, gives
lower bounds under bounded $p$th central moments.  Its stochastic zero-chain
estimator contains a Bernoulli reveal divided by the reveal probability, so
the construction is not uniformly bounded in the hard regime.  Results for
specific algorithms or refined coordinate geometry, such as the AdaGrad and
SGD comparisons of Jiang, Maladkar, and Mokhtari~\cite{JiangEtAl2025}, do not
give an algorithm-independent minimax lower bound for the Euclidean
stationarity criterion considered here.

\begin{table}[t]
\centering
\small
\caption{Closest complexity results and relationship to
Theorem~\ref{thm:main}.}
\vspace{5pt}
\label{tab:literature}
\begin{tabularx}{\textwidth}{@{}>{\raggedright\arraybackslash}p{0.19\textwidth}
>{\raggedright\arraybackslash}p{0.25\textwidth}X@{}}
\toprule
Setting & Representative result & Relation to this manuscript \\
\midrule
Bounded variance & Sharp $\eps^{-4}$ minimax lower bound
\cite{ArjevaniEtAl2023} & The hard oracle is not uniformly bounded; that
paper lists this strengthening as an open extension. \\
Mean-squared smoothness & Sharp $\eps^{-3}$ theory with shared-sample
queries \cite{ArjevaniEtAl2023,FangEtAl2018} & A different oracle assumption
and interaction model; it does not settle uniformly bounded noise with
$K=1$. \\
Bounded $p$th moments & Heavy-tail minimax bounds
\cite{FradinEtAl2026} & Moment control does not provide a common uniform
support bound; the hard estimator uses rare amplified reveals. \\
Algorithm-specific or refined geometry & AdaGrad/SGD lower bounds under
coordinate structure \cite{JiangEtAl2025} & Not a lower bound for every
randomized adaptive method in the present Euclidean model. \\
This manuscript & $\eps^{-4}$ minimax lower bound & Uniformly bounded noise,
fixed law, fresh samples, $K=1$, and arbitrary randomized adaptive
algorithms. \\
\bottomrule
\end{tabularx}
\end{table}

\section{Notation and problem formulation}
\label{sec:model}

Throughout this paper, we use $\|\cdot\|$ to denote Euclidean norm, $\log$ to denote the natural logarithm, and write $[m]=\{1,\ldots,m\}$.
For $d\in\mathbb N$ and $\Delta,L>0$, define
\begin{equation}
\calF_d(\Delta,L)=
\left\{F:\R^d\to\R:
\begin{array}{l}
F\text{ is continuously differentiable and bounded below},\\
F(0)-\inf_xF(x)\le\Delta,\\
\|\nabla F(x)-\nabla F(y)\|\le L\|x-y\|\quad\forall x,y
\end{array}
\right\}.
\label{eq:function-class}
\end{equation}

\begin{definition}[Fixed-law uniformly bounded oracle]
\label{def:oracle}
For $F:\R^d\to\R$, a stochastic first-order oracle is a probability space
$(\calZ,\calP)$ and a jointly Borel measurable map
$g:\R^d\times\calZ\to\R^d$ satisfying
\begin{equation}
  \E_{Z\sim\calP}g(x,Z)=\nabla F(x)\qquad\forall x\in\R^d.
  \label{eq:unbiased}
\end{equation}
It has uniformly almost-surely bounded noise of radius $\sigma$ if there is
one measurable set $\calZ_0$, common to every $x$, with
$\calP(\calZ_0)=1$ and
\begin{equation}
  \|g(x,z)-\nabla F(x)\|\le\sigma
  \qquad\forall x\in\R^d,\quad \forall z\in\calZ_0.
  \label{eq:bounded-noise}
\end{equation}
Successive queries use fresh independent draws from the same law $\calP$.
\end{definition}

A randomized adaptive $N$-query algorithm on $\R^d$ is represented by a
private seed $U$, independent of the instance and samples, measurable query
rules
\[
 x_t=X_t(U,Y_1,\ldots,Y_{t-1}),\qquad
 Y_t=g(x_t,Z_t),\quad 1\le t\le N,
\]
and an arbitrary measurable output rule
$\widehat x_N=X_{N+1}(U,Y_1,\ldots,Y_N)$.  Here
$Z_1,\ldots,Z_N$ are i.i.d. with law $\calP$.  There is one query per fresh
sample ($K=1$), which means the algorithm cannot evaluate the same $Z_t$ at multiple
points.

For a dimension-indexed family of algorithms, the worst-case expected
stationarity risk is the supremum over all finite dimensions, objectives in
\eqref{eq:function-class}, and oracles in Definition~\ref{def:oracle}. 

\begin{remark}[Exact scope]
\label{rem:scope}
Definition~\ref{def:oracle} imposes unbiasedness, joint measurability, and
uniform bounded error, but it does not require $g(x,z)=\nabla_x f(x,z)$ for a
sample loss $f$, nor continuity or conservativity of $g(\cdot,z)$.  The
oracle constructed below is generally discontinuous as a function of $x$.
Consequently, Theorem~\ref{thm:main} does not by itself settle
the samplewise-smooth, finite-sum, or conservative-oracle variants.
\end{remark}

\section{Main result}

\begin{theorem}[Bounded-noise minimax lower bound]
\label{thm:main}
Set
\begin{equation}
 c_0=2^{-20},
 \qquad
 c=\frac{1}{2560\cdot10^4\cdot2^{39}}.
 \label{eq:constants}
\end{equation}
For every $\Delta,L,\sigma>0$, every
\begin{equation}
 0<\eps\le c_0\min\{\sigma,\sqrt{\Delta L}\},
 \label{eq:accuracy-regime}
\end{equation}
and every positive integer
\begin{equation}
 N\le c\left(\frac{\Delta L}{\eps^2}
 +\frac{\Delta L\sigma^2}{\eps^4}\right),
 \label{eq:budget}
\end{equation}
there exists a finite dimension $d=d(\Delta,L,\sigma,\eps,N)$ such that the
following holds.  For every randomized adaptive algorithm on $\R^d$ making
at most $N$ oracle queries, there exist an
$F\in\calF_d(\Delta,L)$ and a fixed-law oracle satisfying
\eqref{eq:unbiased}--\eqref{eq:bounded-noise} for which
\begin{equation}
 \E_{U,Z_{1:N}}\|\nabla F(\widehat x_N)\|>\eps.
 \label{eq:failure}
\end{equation}
In particular, in the regime \eqref{eq:accuracy-regime}, the
dimension-free minimax query complexity is
\[
 \Omega\!\left(\frac{\Delta L}{\eps^2}
 +\frac{\Delta L\sigma^2}{\eps^4}\right).
\]
The explicit dimension is given in \eqref{eq:dimension} below.
\end{theorem}

\paragraph{Proof roadmap.}
An algorithm using
fewer than $N$ queries may be padded by ignored queries, so we analyze exactly
$N$ response rounds.  The proof has five steps.

\begin{enumerate}[leftmargin=*,itemsep=4pt]
\item Section~\ref{sec:chain} recalls a smooth scalar chain of length
\[
T\asymp \frac{\Delta L}{\eps^2}.
\]
Lemma~\ref{lem:chain} says that if the first unfinished coordinate is $j$,
then the $j$th scalar derivative has magnitude at least $2$.  It also gives
the nearest-neighbor dependence later used to hide one link at a time.

\item Section~\ref{sec:embedding} replaces each scalar coordinate by its
correlation with an unknown sign vector $\Theta_j\in\{-1,1\}^D$.  The
$\log\cosh$ term in \eqref{eq:Ftheta} prevents the derivative of the
$\tanh$ encoding from vanishing at large queries.  Lemmas
\ref{lem:objective} and \ref{lem:certificate} show, respectively, that the
resulting objective belongs to $\calF_d(\Delta,L)$ and that a point with
gradient norm at most $2\eps$ must have completed every encoded link.

\item Section~\ref{sec:oracle} splits the gradient at an unfinished point
into a part determined by already encountered sign vectors and a small
remaining vector that depends on the next sign vector.  It realizes the
remaining vector as the mean of bounded random signs.  Lemma
\ref{lem:oracle} proves uniform boundedness and unbiasedness.  A single
response has KL divergence at most $KD\eps^2/\sigma^2$ from the same
response with unbiased signs.

\item Section~\ref{sec:information} draws all sign vectors independently.
An independent query has probability at most $2e^{-D/128}$ of having enough
correlation to cross a new link, while the responses collected before that
crossing reveal limited information about the relevant sign vector.  A
fixed-length comparison argument makes this statement valid for adaptive
queries.  It shows that crossing one link with appreciable probability
requires
\[
n_0\asymp \frac{\sigma^2}{\eps^2}
\]
responses whose laws depend on that link.

\item The responses charged to different links are disjoint.  Hence fewer
than $Tn_0/2$ total queries complete all $T$ links with probability at most
$9/128$.  Lemma~\ref{lem:certificate} then gives the desired lower bound on
the expected gradient norm.  Section~\ref{sec:parameters} verifies the
constants and fixes one deterministic choice of the random sign vectors.
\end{enumerate}

The first step follows the deterministic zero-chain construction of Carmon
et al.~\cite{CarmonEtAl2020}.  The new ingredient relative to the
bounded-variance lower bound of Arjevani et al.~\cite{ArjevaniEtAl2023} is
Step 3: their rare unbiased reveal must have large amplitude, whereas the
biased-sign construction keeps every realization of the oracle error bounded and hides the next
link through small KL information instead.

\section{The hard scalar chain}
\label{sec:chain}

This section defines the scalar chain used as the base of the construction.
The chain has many coordinates, but its local form allows information to move
from coordinate $j$ only to coordinate $j+1$.  We use three consequences:
the chain has gap $O(T)$, has uniformly bounded first and second derivatives,
and has a nonzero derivative at its first unfinished coordinate.

Define
\[
\Psi(s)=
\begin{cases}
0,&s\le \tfrac12,\\
\exp\!\left(1-(2s-1)^{-2}\right),&s>\tfrac12,
\end{cases}
\qquad
\Phi(s)=\sqrt e\int_{-\infty}^s e^{-t^2/2}\,dt .
\]
For $s\in\R^T$, let
\begin{align}
\bar f_T(s)
&=-\Phi(s_1)+\sum_{i=2}^T
\bigl[\Psi(-s_{i-1})\Phi(-s_i)-\Psi(s_{i-1})\Phi(s_i)\bigr],
\nonumber\\
f_T(q)&=\bar f_T(2q).
\label{eq:chain}
\end{align}

\begin{lemma}[Chain bounds]
\label{lem:chain}
For every $T\ge1$,
\begin{equation}
 f_T(0)-\inf f_T\le20T,
 \qquad
 \|\nabla f_T(q)\|_\infty\le G:=64,
 \qquad
 \operatorname{Lip}(\nabla f_T)\le\ell:=2400.
 \label{eq:chain-bounds}
\end{equation}
If $j$ is the least index satisfying $|q_j|\le1/2$, then
\begin{equation}
 b_j(q):=\partial_j f_T(q)\le-2.
 \label{eq:chain-slope}
\end{equation}
If
$p=\max(\{i:|q_i|>1/4\}\cup\{0\})$, then $b_i(q)=0$ for every
$i\ge p+2$.  If $p<T$, the dependence of
$b_1,\ldots,b_{p+1}$ on $q_{p+1}$ occurs only in $b_p,b_{p+1}$,
with the first item omitted when $p=0$.
\end{lemma}

\begin{proof}
Put $u=(2s-1)^{-2}$ for $s>1/2$.  Direct differentiation gives
\[
 \Psi'(s)=4e e^{-u}u^{3/2},
 \qquad
 \Psi''(s)=8e e^{-u}u^2(2u-3).
\]
Every right derivative is $e e^{-u}$ times a polynomial in $\sqrt u$, and
$u^m e^{-u}\to0$ as $u\to\infty$.  Thus all right derivatives tend to zero
at the cutoff and the flat extension is smooth.  Using
$\max_{u\ge0}u^ke^{-u}=k^ke^{-k}$,
\[
 \Psi'\le3\sqrt{6/e}<5,
 \qquad
 |\Psi''|\le16eu^3e^{-u}+24eu^2e^{-u}
 \le\frac{432}{e^2}+\frac{96}{e}<100.
\]
Consequently,
\begin{equation}
 0\le\Psi<3,
 \qquad 0\le\Psi'<5,
 \qquad |\Psi''|<100.
 \label{eq:Psi-bounds}
\end{equation}
Also
\begin{equation}
 0<\Phi<\sqrt{2\pi e}<5,
 \qquad 0<\Phi'<2,
 \qquad |\Phi''|\le1.
 \label{eq:Phi-bounds}
\end{equation}

For any scalar $s$, at most one of $\Psi^{(k)}(s)$ and
$\Psi^{(k)}(-s)$ is nonzero for $k=0,1,2$.  Differentiation gives
\begin{align}
\partial_j\bar f_T(s)
={}&-\ind_{\{j=1\}}\Phi'(s_1)\nonumber\\
&-\ind_{\{j\ge2\}}
\bigl[\Psi(-s_{j-1})\Phi'(-s_j)+\Psi(s_{j-1})\Phi'(s_j)\bigr]
\nonumber\\
&-\ind_{\{j<T\}}
\bigl[\Psi'(-s_j)\Phi(-s_{j+1})+\Psi'(s_j)\Phi(s_{j+1})\bigr].
\label{eq:first-derivative-chain}
\end{align}
Every displayed term is nonpositive.  An incoming bracket has magnitude at
most $6$, an outgoing bracket at most $25$, and the root term at most $2$.
Hence $\|\nabla\bar f_T\|_\infty\le31$ and
$\|\nabla f_T\|_\infty\le62<64$.

The Hessian of $\bar f_T$ is tridiagonal.  Its diagonal entries are
\begin{align*}
\partial_{jj}^2\bar f_T
={}&-\ind_{\{j=1\}}\Phi''(s_1)\\
&+\ind_{\{j\ge2\}}
\bigl[\Psi(-s_{j-1})\Phi''(-s_j)-\Psi(s_{j-1})\Phi''(s_j)\bigr]\\
&+\ind_{\{j<T\}}
\bigl[\Psi''(-s_j)\Phi(-s_{j+1})-\Psi''(s_j)\Phi(s_{j+1})\bigr].
\end{align*}
Its adjacent entries have the form
\[
 \Psi'(-s_j)\Phi'(-s_{j+1})-\Psi'(s_j)\Phi'(s_{j+1}),
\]
and all other off-diagonal entries vanish.  The root, incoming, and outgoing
diagonal contributions are bounded by $1$, $3$, and $500$; the root and
incoming terms never occur together.  An adjacent entry has magnitude at
most $10$.  Thus every diagonal entry has magnitude at most $503$, every
absolute row sum is at most $523$, and
$\|\nabla^2\bar f_T\|_{\op}\le523$.  The substitution $s=2q$ multiplies
the Hessian by four, so $\|\nabla^2f_T\|_{\op}\le2092<2400$.

For the gap, at most one summand in each bracket of \eqref{eq:chain} is
nonzero, and each product has magnitude at most $3\cdot5=15$.  Therefore
$\bar f_T(s)\ge-5-15(T-1)$.  Since $\bar f_T(0)=-\Phi(0)<0$ and the scaling
$s=2q$ is onto,
\[
 f_T(0)-\inf f_T\le5+15(T-1)\le15T<20T.
\]

For \eqref{eq:chain-slope}, if $j=1$, the root term gives
$-2\Phi'(2q_1)\le-2$.  If $j>1$, then $|2q_{j-1}|>1$.  Exactly one
incoming term is active, its gate is at least $\Psi(1)=1$, and
$\Phi'(2q_j)=\sqrt e e^{-2q_j^2}\ge1$.  After the factor two from
$f_T(q)=\bar f_T(2q)$, this term is at most $-2$.  All outgoing terms have
the same sign.

Finally, if $i\ge p+2$, then $|q_{i-1}|$ and $|q_i|$ are at most $1/4$,
so every incoming and outgoing term in $b_i$ vanishes.  The nearest-neighbor
form of \eqref{eq:chain} also shows that $q_{p+1}$ can enter only
$b_p,b_{p+1}$.  This proves the lemma.
\end{proof}

The scalar chain alone is hard only when its coordinate system cannot be
guessed in advance.  The next section hides each scalar coordinate in a
random sign block and transfers the derivative lower bound above back to the
original variable $x$.

\section{Encoding the chain and preventing escape}
\label{sec:embedding}

We now construct the objective used in the theorem. In  Equation
\eqref{eq:q-code} below, we encode scalar coordinate $q_j$ as a correlation between a
query block and an unknown sign vector $\theta_j$.  The first term of
\eqref{eq:Ftheta} then places the scalar chain on these encoded coordinates, while the
second term is needed because the derivative of $\tanh$ approaches zero for
large inputs; without this term, an algorithm could make the encoded chain
appear stationary simply by taking $\|x\|$ large.

Fix
\begin{equation}
 C=10000,
 \qquad A=\frac{16C\eps^2}{L},
 \qquad \delta=\frac{A}{4\eps}=\frac{4C\eps}{L},
 \qquad T=\left\lfloor\frac{\Delta L}{320C\eps^2}\right\rfloor.
 \label{eq:ATdelta}
\end{equation}
Condition \eqref{eq:accuracy-regime} will imply $T\ge1$.  Let $D\ge1$ be an
integer, and let $\theta=(\theta_1,\ldots,\theta_T)$, where
$\theta_j\in\{-1,1\}^D$.  Write
$x=(x_1,\ldots,x_T)\in(\R^D)^T$, set
\begin{equation}
 y_{j\ell}=\frac{\sqrt D\,x_{j\ell}}{\delta},
 \qquad
 q_j(x)=\frac2D\sum_{\ell=1}^D\theta_{j\ell}\tanh y_{j\ell},
 \label{eq:q-code}
\end{equation}
and define
\begin{equation}
 F_\theta(x)=A f_T(q(x))
 +\frac{8A}{D}\sum_{j=1}^T\sum_{\ell=1}^D\log\cosh y_{j\ell}.
 \label{eq:Ftheta}
\end{equation}

\begin{lemma}[Objective bounds]
\label{lem:objective}
The function $F_\theta$ is continuously differentiable, bounded below,
globally $L$-smooth, and satisfies
\begin{equation}
 F_\theta(0)-\inf F_\theta\le\Delta.
 \label{eq:objective-gap}
\end{equation}
\end{lemma}

\begin{proof}
The derivatives of \eqref{eq:q-code} are
\begin{equation}
 \partial_{x_{j\ell}}q_j
 =\frac{2\theta_{j\ell}}{\delta\sqrt D}\sech^2y_{j\ell},
 \qquad
 \partial^2_{x_{j\ell}}q_j
 =-\frac4{\delta^2}\theta_{j\ell}\tanh y_{j\ell}\sech^2y_{j\ell}.
 \label{eq:q-derivatives}
\end{equation}
Thus $\|Dq\|_{\op}\le2/\delta$ and
$\|\nabla^2q_j\|_{\op}\le4/\delta^2$.  The latter Hessians occupy disjoint
blocks.  The Hessian chain rule, including the added $\log\cosh$ term, gives
\[
 \nabla^2F_\theta
 =A Dq^\top\nabla^2f_T(q)Dq
 +A\sum_{j=1}^Tb_j(q)\nabla^2q_j
 +\frac{8A}{\delta^2}\operatorname{diag}(\sech^2y_{j\ell}).
\]
The first term has norm at most $4A\ell/\delta^2$, the second at most
$4AG/\delta^2$, and the last at most $8A/\delta^2$.  Lemma~\ref{lem:chain}
therefore implies
\begin{equation}
 \operatorname{Lip}(\nabla F_\theta)
 \le\frac{A}{\delta^2}(4\ell+4G+8)
 =\frac{L}{C}(4\ell+4G+8)\le L.
 \label{eq:smoothness}
\end{equation}
The scalar chain is bounded below and the added term is nonnegative, so
$F_\theta$ is bounded below.  Since $q(0)=0$ and this term vanishes at zero,
\[
 F_\theta(0)-\inf F_\theta
 \le A(f_T(0)-\inf f_T)
 \le20AT
 \le20\frac{16C\eps^2}{L}\frac{\Delta L}{320C\eps^2}
 =\Delta.
\]
\end{proof}

We next verify that the added $\log\cosh$ term does more than preserve
smoothness: together with the chain derivative, it keeps the gradient large
whenever any encoded coordinate is unfinished, including at points of
arbitrarily large norm.

\begin{lemma}[An unfinished coordinate forces a large gradient]
\label{lem:certificate}
If some coordinate obeys $|q_j(x)|\le1/2$, then
\begin{equation}
 \|\nabla F_\theta(x)\|>2\eps.
 \label{eq:certificate}
\end{equation}
\end{lemma}

\begin{proof}
Take the least such $j$ and let $b=b_j(q(x))$.  Lemma~\ref{lem:chain} gives
$|b|\ge2$.  The $j$th block gradient is
\begin{equation}
 \nabla_{x_j}F_\theta(x)
 =\frac{A}{\delta\sqrt D}
 \left[2b\theta_{j\ell}\sech^2y_{j\ell}+8\tanh y_{j\ell}
 \right]_{\ell=1}^D.
 \label{eq:block-gradient}
\end{equation}
Let $r_\ell=\theta_{j\ell}\tanh y_{j\ell}$ and
$z_b(r)=2b(1-r^2)+8r$. Since $\theta_{j\ell}^2=1$, the bracket in \eqref{eq:block-gradient} equals
$\theta_{j\ell}z_b(r_\ell)$.  For $b>0$, the roots are
\[
 r_b=\frac{2-\sqrt{b^2+4}}b,
 \qquad
 R_b=\frac{2+\sqrt{b^2+4}}b.
\]
Here $r_b\in(-1,0)$ and $R_b>1$.  The case $b<0$ follows from
$z_b(r)=-z_{|b|}(-r)$.  Thus $z_b$ has one root in $[-1,1]$, and
\[
 m(t)=\frac{t}{\sqrt{t^2+4}+2},
 \qquad
 m'(t)=\frac{2+4/\sqrt{t^2+4}}{(\sqrt{t^2+4}+2)^2}>0.
\]
Consequently,
\begin{equation}
 |r_b|=\frac{|b|}{\sqrt{b^2+4}+2}=m(|b|)\ge m(2)=\sqrt2-1.
 \label{eq:root-lower}
\end{equation}
For $b>0$, factorization gives, for $-1\le r\le1$ and $r\ne r_b$,
\[
 \frac{|z_b(r)|}{|r-r_b|}
 =2b(R_b-r)\ge2b(R_b-1)
 =2(2+\sqrt{b^2+4}-b)\ge4.
\]
Symmetry handles $b<0$, and therefore
\begin{equation}
 |z_b(r)|\ge4|r-r_b|,
 \qquad -1\le r\le1.
 \label{eq:z-lower}
\end{equation}
Since $D^{-1}\sum_\ell r_\ell=q_j/2$ by definition, its absolute value is at most $1/4$.
Jensen's inequality and \eqref{eq:root-lower}--\eqref{eq:z-lower} yield
\begin{align}
\left(\frac1D\sum_{\ell=1}^Dz_b(r_\ell)^2\right)^{1/2}
&\ge4\left(\frac1D\sum_{\ell=1}^D(r_\ell-r_b)^2\right)^{1/2}
\nonumber\\
&\ge4\left|\frac1D\sum_{\ell=1}^Dr_\ell-r_b\right|
\ge4\sqrt2-5>\frac12.
\label{eq:block-lower}
\end{align}
Now $A/\delta=4\eps$.  Combining \eqref{eq:block-gradient} and
\eqref{eq:block-lower} proves \eqref{eq:certificate}.
\end{proof}

Lemma~\ref{lem:certificate} reduces
optimization to a sequential task: a small-gradient output must cross every
one of the $T$ encoded links.  It remains to construct a bounded oracle that
makes learning the next sign vector slow.

\section{A uniformly bounded oracle with fixed law}
\label{sec:oracle}

Fix a query $x$ whose last crossed coordinate is $p(x)$, as defined below in \eqref{eq:last-crossed}.
The next unresolved sign vector is then $\theta_{p(x)+1}$.  The purpose of
this section is to write
\[
\begin{aligned}
\nabla F_\theta(x)
&=\text{a vector determined by earlier sign vectors}\\
&\quad+\text{a small vector depending on the next sign vector}.
\end{aligned}
\]
and to construct an unbiased bounded estimate of the second vector.  The
decomposition is established in \eqref{eq:B-plus-h}; the bounded variable is defined in \eqref{eq:bounded-response}.

Write the block gradient as $\nabla_iF_\theta=G_i+C_i$, where
\begin{equation}
 G_i=\frac{8A}{\delta\sqrt D}
 [\tanh y_{i\ell}]_{\ell=1}^D,
 \qquad
 C_i=\frac{2A}{\delta\sqrt D}b_i(q)
 [\theta_{i\ell}\sech^2y_{i\ell}]_{\ell=1}^D.
 \label{eq:gradient-decomposition}
\end{equation}
Let $\iota_i:\R^D\to\R^{TD}$ be the canonical embedding into block $i$,
and set $\calG=(G_1,\ldots,G_T)$.  The vectors $C_i$ and their variants
remain $D$-dimensional block vectors; we apply $\iota_i$ explicitly whenever
they enter a vector in $\R^{TD}$.

Let
\begin{equation}
 p(x)=\max\bigl(\{i:|q_i(x)|>1/4\}\cup\{0\}\bigr).
 \label{eq:last-crossed}
\end{equation}
If $p(x)=T$, define $g_\theta(x,z)=\nabla F_\theta(x)$.  Suppose
$p(x)<T$, and put $j=p(x)+1$.  For $j\ge2$, define
$q^{(j\leftarrow0)}(x)\in\R^T$ coordinatewise by
\[
 \bigl(q^{(j\leftarrow0)}(x)\bigr)_k
 =
 \begin{cases}
  0, & k=j,\\
  q_k(x), & k\ne j.
 \end{cases}
\]
Thus $q^{(j\leftarrow0)}(x)$ is obtained from $q(x)$ by replacing only
its $j$th coordinate by zero.  Define the $D$-dimensional block vector
\[
 \widetilde C_{j-1}^{(j)}(x)
 =
 \frac{2A}{\delta\sqrt D}
 b_{j-1}\!\left(q^{(j\leftarrow0)}(x)\right)
 \left(
 \theta_{j-1}\od\sech^2 y_{j-1}
 \right).
\]
Now set
\begin{align}
 B_j
 &=
 \calG
 +\sum_{i\le j-2}\iota_i(C_i)
 +\ind_{\{j\ge2\}}
 \iota_{j-1}\!\left(\widetilde C_{j-1}^{(j)}\right),
 \label{eq:Bj}\\
 h_j
 &=
 \ind_{\{j\ge2\}}
 \iota_{j-1}\!\left(
 C_{j-1}-\widetilde C_{j-1}^{(j)}
 \right)
 +\iota_j(C_j).
 \label{eq:hj}
\end{align}
Here $\calG$ contains the full contribution of the $\log\cosh$ term.
The vector $B_j$ is determined by
$\theta_1,\ldots,\theta_{j-1}$, while $h_j$ is the remaining part of
the gradient that may depend on $\theta_j$.

To verify this decomposition, suppose $p=j-1$.  Then every $q_i$ with
$i\ge j$ satisfies $|q_i|\le1/4$, and hence
$\Psi(\pm2q_i)=\Psi'(\pm2q_i)=0$.  The nearest-neighbor property in
Lemma~\ref{lem:chain} gives $C_i=0$ for $i\ge j+1$, removes
$q_{j+1}$ from $b_j$, and shows that only $C_{j-1}$ among the preceding
blocks can depend on $q_j$.  For $j=1$, the outgoing part of $b_1$
vanishes; for $j=T$, there is no outgoing term.  Consequently,
\begin{equation}
 \nabla F_\theta(x)=B_j(x)+h_j(x).
 \label{eq:B-plus-h}
\end{equation}

Moreover, $h_j$ is supported on block $1$ when $j=1$ and on blocks
$j-1,j$ when $j\ge2$.  Its nonzero blocks are
\begin{align*}
 (h_j)_{j-1}
 &=
 \frac{2A}{\delta\sqrt D}
 \left[
 b_{j-1}(q(x))
 -
 b_{j-1}\!\left(q^{(j\leftarrow0)}(x)\right)
 \right]
 \left(
 \theta_{j-1}\od\sech^2 y_{j-1}
 \right),
 \qquad j\ge2,\\
 (h_j)_j
 &=
 \frac{2A}{\delta\sqrt D}
 b_j(q(x))
 \left(
 \theta_j\od\sech^2 y_j
 \right).
\end{align*}
Since
\[
 |b_j(q(x))|\le G
\]
and
\[
 \left|
 b_{j-1}(q(x))
 -
 b_{j-1}\!\left(q^{(j\leftarrow0)}(x)\right)
 \right|
 \le 2G,
\]
while $A/\delta=4\eps$, every coordinate of $h_j$ satisfies
\begin{equation}
 |(h_j)_k|
 \le
 \frac{H\eps}{\sqrt D},
 \qquad
 H:=1024,
 \qquad
 \|h_j\|\le\sqrt2H\eps.
 \label{eq:h-bounds}
\end{equation}
Indeed, a current-block coordinate is at most
$8G\eps/\sqrt D=512\eps/\sqrt D$, while a backward-difference coordinate
is at most $16G\eps/\sqrt D=1024\eps/\sqrt D$.

Take the fixed sample space
\[
 \calZ=[0,1]^{TD}
\]
with product Lebesgue measure.  Let $r=\sigma/2$, and define
\[
 \mathcal I_1=\{1\}\times[D],
 \qquad
 \mathcal I_j=\{j-1,j\}\times[D]\quad(j\ge2),
 \qquad m_j=|\mathcal I_j|.
\]
On coordinates $k=(i,\ell)\in\mathcal I_j$, define
\begin{equation}
 \beta_k(x)=\frac{\sqrt{m_j}(h_j(x))_k}{r},
 \qquad
 S_{i\ell}(x,z)=2\ind\!\left\{
 z_{i\ell}\le\frac{1+\beta_{i\ell}(x)}2\right\}-1,
 \qquad
 R_{i\ell}(x,z)=\frac r{\sqrt{m_j}}S_{i\ell}(x,z),
 \label{eq:bounded-response}
\end{equation}
and put $R=0$ outside $\mathcal I_j$.  For $p(x)<T$, we then define the response to the gradient oracle as
\begin{equation}
 g_\theta(x,z)=B_j(x)+R(x,z),\qquad j=p(x)+1.
 \label{eq:oracle}
\end{equation}

\begin{lemma}[Oracle validity]
\label{lem:oracle}
If
\begin{equation}
 \eps\le2^{-20}\sigma,
 \label{eq:noise-regime}
\end{equation}
then \eqref{eq:oracle}, together with the exact-gradient branch $p=T$, is
jointly Borel measurable, unbiased at every $x$, and obeys
\begin{equation}
 \|g_\theta(x,z)-\nabla F_\theta(x)\|<\sigma
 \qquad\text{for every }x\in\R^{TD},\ z\in\calZ.
 \label{eq:uniform-bound}
\end{equation}
Moreover, Fresh samples $Z_t\overset{\mathrm{i.i.d.}}{\sim}\calP$ are used at successive queries, and the sample law $\calP$ remains fixed throughout the interaction.
\end{lemma}

\begin{proof}
From \eqref{eq:h-bounds} and \eqref{eq:noise-regime},
\begin{equation}
 \|h_j\|<\sigma/4.
 \label{eq:h-small}
\end{equation}
Since $m_j\le2D$, every active coordinate satisfies
\[
 |\beta_k|
 \le\frac{\sqrt{2D}}r\frac{H\eps}{\sqrt D}
 \le\frac{2\sqrt2H\eps}{\sigma}\le\frac12.
\]
Thresholding a uniform coordinate in \eqref{eq:bounded-response} gives
$\E S_k=\beta_k$, hence $\E R=h_j$.  Equations
\eqref{eq:B-plus-h} and \eqref{eq:oracle} prove unbiasedness.  Every
realization satisfies $\|R\|=r$, and therefore
\[
 \|g_\theta(x,z)-\nabla F_\theta(x)\|
 \le\|R\|+\|h_j\|<3\sigma/4.
\]
It remains to verify the uniformity, measurability, and fixed-sample-law
requirements in Definition~\ref{def:oracle}.  The bound above holds for
every $z\in\calZ$ when $p(x)<T$, while the oracle returns the exact
gradient when $p(x)=T$; hence the common probability-one set can be taken
to be all of $\calZ$.  All ingredients of $g_\theta$ are Borel, and the
sets $\{x:p(x)=p\}$ are Borel, so the finite branch definition makes
$g_\theta$ jointly Borel measurable.  Finally, fixing product Lebesgue
measure $\calP$ on $\calZ$ and drawing fresh
$Z_t\overset{\mathrm{i.i.d.}}{\sim}\calP$ at successive queries gives a
single sample law throughout the interaction, although the response
distribution may vary with the query.
\end{proof}

We now define the two response distributions used in the information
argument.  Let $p<T$, let $x$ satisfy $p(x)=p$.  For every
Borel set $\mathcal A\subseteq\R^{TD}$, define
\[
 \mathsf P_p^\theta(x,\mathcal A)
 :=\Pr_Z\bigl(g_\theta(x,Z)\in\mathcal A\bigr).
\]
Thus $\mathsf P_p^\theta(x,\cdot)$ is simply the law of the next oracle
response after the current query and the past history have been fixed.  Let
$\mathsf Q_p^{\theta_{\le p}}(x,\cdot)$ be the same distribution \emph{except that
the signs on $\mathcal I_{p+1}$ are fair, rather than having the biases
$\beta_k(x)$ in \eqref{eq:bounded-response}}.  Its base vector is still
$B_{p+1}(x)$, so after conditioning on the previously encountered sign
vectors $\theta_{\le p}$, this reference distribution is independent of
$\theta_{p+1}$.

For $|u|\le1/2$,
\[
 \frac12\bigl[(1+u)\log(1+u)+(1-u)\log(1-u)\bigr]\le u^2.
\]
Indeed, the left side and its first derivative vanish at zero, and its second
derivative is $1/(1-u^2)\le4/3$.  Independence of the sign coordinates and
\eqref{eq:h-bounds} imply
\begin{align}
\KL\!\left(
\mathsf P_p^\theta(x,\cdot)
\,\middle\|\,
\mathsf Q_p^{\theta_{\le p}}(x,\cdot)
\right)
&\le\sum_{k\in\mathcal I_j}\beta_k^2
=\frac{m_j}{r^2}\|h_j\|^2
\nonumber\\
&\le16H^2D\frac{\eps^2}{\sigma^2}
=K D\frac{\eps^2}{\sigma^2},
\qquad K:=16H^2=2^{24}.
\label{eq:one-response-kl}
\end{align}
The left side is a conditional KL divergence: in an adaptive interaction,
the past fixes $x$ and the old sign vectors before this one-step comparison
is made.

\begin{remark}
This is where the bounded-noise construction departs from the earlier
bounded-variance lower bounds.  In prior constructions, a rare reveal with probability $\rho$ must
scale a mean vector $h$ to $h/\rho$ when it occurs, which violates a uniform
bound when $\rho$ is small.  Here the mean $h_j$ is spread across many signs and the oracle error remains uniformly bounded almost surely. Hardness comes from the small KL
divergence above rather than from a rare large response.  
\end{remark}

The distributions
$\mathsf P_p^\theta$ and $\mathsf Q_p^{\theta_{\le p}}$ have so far been
defined only at points whose value of \eqref{eq:last-crossed} is $p$.
The next section extends the same formulas to every query point so that they
can be used in a fixed-length comparison.

\section{Adaptive information accumulation and the query lower bound}
\label{sec:information}

The purpose of this section is to turn the one-response KL bound
\eqref{eq:one-response-kl} into a lower bound on the number of oracle
responses required to cross one link of the chain.  The actual oracle is
not changed in this argument.  We introduce a reference response
distribution only for comparison: for a fixed coordinate \(i\), the
reference responses are constructed so that they do not depend on
\(\Theta_i\).  If an algorithm could cross coordinate \(i\) after only a
few relevant responses, it would therefore distinguish the actual and
reference response distributions with substantial probability.

The comparison is quantified by the conditional KL chain rule.  Consider
two adaptive experiments with joint laws \(\mathbf P\) and \(\mathbf Q\)
that use the same algorithm and differ only in some of their conditional
response distributions.  For this paragraph, let
\(\mathcal F_{s-1}\) denote all variables fixed before the \(s\)-th response
\(Y_s\) is sampled, including the sign vector under consideration, the
initial variables on which we have conditioned, and the preceding
responses.  The chain rule gives
\[
  \KL(\mathbf P\|\mathbf Q)
  =
  \mathbb E_{\mathbf P}
  \left[
    \sum_s
    \KL\!\left(
      \mathbf P(Y_s\in\cdot\mid\mathcal F_{s-1})
      \,\middle\|\,
      \mathbf Q(Y_s\in\cdot\mid\mathcal F_{s-1})
    \right)
  \right].
\]
The queries make no separate contribution to this sum because, conditional
on \(\mathcal F_{s-1}\), the next query is determined by the algorithm.
In particular, each history sampled under the outer
\(\mathbf P\)-expectation determines a possibly different query point
\(x_s\).

This observation explains why an extension of the response distributions
is necessary.  Section~\ref{sec:oracle} defined
\(\mathsf P_p^\theta(x,\cdot)\) and
\(\mathsf Q_p^{\theta_{\le p}}(x,\cdot)\) only when
\(p_\theta(x)=p\).  In the comparison above, however, the integer \(p\)
used to select a response distribution and the adaptively generated point
\(x_s\) need not satisfy \(p_\theta(x_s)=p\).  To control every conditional
KL term, we therefore extend both response distributions to every pair
\((p,x)\), with \(0\le p<T\), and prove the uniform bound
\[
  \KL\!\left(
    \mathsf P_p^\theta(x,\cdot)
    \,\middle\|\,
    \mathsf Q_p^{\theta_{\le p}}(x,\cdot)
  \right)
  \le KD\frac{\eps^2}{\sigma^2}.
\]
The extended response agrees with the actual oracle whenever
\(p_\theta(x)=p\).  Thus it can be used at arbitrary query points in the
KL calculation without altering the actual oracle interaction on the
events where the two are later compared.

The remainder of the section instantiates this argument.  We begin by constructing an extension of the response
distributions, and establish their uniform KL bound. Next, we define an interaction process whose crossed coordinate can increase by at most one and prove that it agrees with the actual interaction except on an event of
small probability. For each fixed \(i\), we then compare two
fixed-length continuations: their conditional response distributions
differ at no more than \(n\) positions, so the chain rule bounds their
total KL divergence by
\[
  nKD\frac{\eps^2}{\sigma^2}.
\]
Under the reference continuation, the queries are independent of
\(\Theta_i\), and the correlation bound therefore makes early crossing
unlikely.  Finally, we sum the required response counts over
\(i=1,\ldots,T\) to obtain the query lower bound.
\\

Define
\begin{equation}
 n_0=\left\lfloor\frac{\sigma^2}{2^{14}K\eps^2}\right\rfloor,
 \qquad
 D=\left\lceil\max\left\{2^{15},
 256\log\bigl(16(N+1)T\bigr)\right\}\right\rceil.
 \label{eq:n0D}
\end{equation}
Condition \eqref{eq:accuracy-regime} and the definition
\eqref{eq:ATdelta} imply $T\ge1$ and
$\sigma^2/\eps^2\ge2^{40}$, so that $n_0 \geq 4$.

Let $\Theta_1,\ldots,\Theta_T$ be drawn independently and uniformly from
$\{-1,1\}^D$, independently of the algorithm's private seed $U$ and the
oracle samples $Z_1,\ldots,Z_N$.  For a deterministic sign sequence
$\theta=(\theta_1,\ldots,\theta_T)$, define
\begin{equation}
 a_i(x):=[\tanh(\sqrt D\,x_{i\ell}/\delta)]_{\ell=1}^D,
 \qquad
 q_i^\theta(x):=\frac2D\langle\theta_i,a_i(x)\rangle,
 \label{eq:a-code}
\end{equation}
Write
\[
 q^\theta(x):=(q_1^\theta(x),\ldots,q_T^\theta(x)),
\]
and define
\begin{equation}
 p_\theta(x):=
 \max\bigl(\{i\in[T]:|q_i^\theta(x)|>1/4\}\cup\{0\}\bigr).
 \label{eq:random-last-crossed}
\end{equation}
Thus $q_i^\theta$ and $p_\theta$ are exactly the functions $q_i$ and $p$
from Sections~\ref{sec:embedding} and~\ref{sec:oracle}, with the dependence
on the sign sequence explicitly displayed.  We use $q_i^\Theta$ and $p_\Theta$ for
their random-sign versions.  We also write
$\theta_{<i}=(\theta_1,\ldots,\theta_{i-1})$ and
$\theta_{\le i}=(\theta_1,\ldots,\theta_i)$; the corresponding notation is
used for $\Theta$, and $\theta_{\le0}$ denotes the empty tuple.

\subsection{A correlation bound for an independent sign vector}

The first lemma is the only concentration estimate needed below.  Its
random-vector formulation is useful because an adaptive query may depend on
everything observed so far, provided that this information does not depend
on the particular sign vector appearing in the lemma.

\begin{lemma}[Independent-sign correlation]
\label{lem:sign-correlation}
Let $\Xi$ be uniform on $\{-1,1\}^D$, and let $V$ be an
$[-1,1]^D$-valued random vector independent of $\Xi$.  Then
\begin{equation}
 \Prob\!\left(
 \left|\frac2D\langle\Xi,V\rangle\right|>\frac14
 \right)
 \le 2e^{-D/128}.
 \label{eq:sign-tail}
\end{equation}
\end{lemma}

\begin{proof}
Condition on $V=v$.  For every $\lambda\in\R$, independence of the
coordinates of $\Xi$ and the inequality $\cosh u\le e^{u^2/2}$ give
\[
 \E\exp\!\left(\lambda\langle\Xi,v\rangle\right)
 =\prod_{\ell=1}^D\cosh(\lambda v_\ell)
 \le\exp\!\left(\frac{\lambda^2}{2}
                 \sum_{\ell=1}^Dv_\ell^2\right)
 \le e^{\lambda^2D/2}.
\]
Chernoff's inequality at threshold $D/8$, optimized at
$\lambda=1/8$, yields
\[
 \Prob\bigl(\langle\Xi,v\rangle>D/8\mid V=v\bigr)
 \le e^{-D/128}.
\]
Apply the same bound to $-\Xi$ and then average over $V$ yields the desired inequality.
\end{proof}

\subsection{Response distributions at arbitrary queries}

In Section~\ref{sec:oracle}, we defined
$\mathsf P_p^\theta(x,\cdot)$ and
$\mathsf Q_p^{\theta_{\le p}}(x,\cdot)$ for
$p_\theta(x)=p$.  The next definition extends these distributions to every
pair $(p,x)$: it specifies the response that would be generated at $x$ if
$p+1$ were treated as the next coordinate whose sign vector may affect the
response, even when the actual value of $p_\theta(x)$ is different from
$p$.  This extension allows the comparison interaction in the next
subsection to choose its response index without first conditioning on
$p_\theta(x)=p$, while it is also guaranteed that the extended
response is identical to the actual oracle response whenever
$p_\theta(x)=p$.  The following lemma proves this agreement together with
the uniform KL bound.

\begin{definition}[Extended response maps]
\label{def:extended-response}
Fix a deterministic $\theta$.  We use the explicit notation
\[
 \calG(x):=(G_1(x),\ldots,G_T(x)).
\]
For $0\le p<T$, set
\[
 q_\theta^{[p]}(x)
 :=(q_1^\theta(x),\ldots,q_{p+1}^\theta(x),0,\ldots,0).
\]
For $1\le i\le p+1$, define the full-dimensional vector
\begin{equation}
 \mathbf C_i^{[p],\theta}(x)
 :=\iota_i\!\left(
 \frac{2A}{\delta\sqrt D}\,
 b_i(q_\theta^{[p]}(x))
 [\theta_{i\ell}\sech^2y_{i\ell}]_{\ell=1}^D
 \right)\in\R^{TD}.
 \label{eq:extended-C}
\end{equation}
For $1\le p<T$, also define
\begin{equation}
 \widetilde{\mathbf C}_p^{[p],\theta}(x)
 :=\iota_p\!\left(
 \frac{2A}{\delta\sqrt D}\,
 b_p(q_\theta^{[p-1]}(x))
 [\theta_{p\ell}\sech^2y_{p\ell}]_{\ell=1}^D
 \right)\in\R^{TD}.
 \label{eq:extended-C-tilde}
\end{equation}
Now let
\begin{align}
 \Gamma_p^\theta(x)
 &:=\calG(x)+\sum_{i=1}^{p+1}\mathbf C_i^{[p],\theta}(x),
 \nonumber\\
 \bar B_0^\theta(x)&:=\calG(x),
 \nonumber\\
 \bar B_p^\theta(x)
 &:=\calG(x)+\sum_{i=1}^{p-1}\mathbf C_i^{[p],\theta}(x)
       +\widetilde{\mathbf C}_p^{[p],\theta}(x),
       \qquad 1\le p<T,
 \nonumber\\
 H_p^\theta(x)&:=\Gamma_p^\theta(x)-\bar B_p^\theta(x).
 \label{eq:extended-decomposition}
\end{align}
Thus $\Gamma_p^\theta(x),\bar B_p^\theta(x),H_p^\theta(x) \in\R^{TD}$.
\\

For $k\in\mathcal I_{p+1}$, define
\begin{align}
 \beta_{p,k}^\theta(x)
 &:=\frac{\sqrt{m_{p+1}}\,(H_p^\theta(x))_k}{r},
 \nonumber\\
 S_{p,k}^\theta(x,z)
 &:=2\ind\!\left\{z_k\le
       \frac{1+\beta_{p,k}^\theta(x)}2\right\}-1,
 \nonumber\\
 S_{p,k}^0(z)&:=2\ind\{z_k\le1/2\}-1,
 \label{eq:extended-signs}
\end{align}
where we recall from Section~\ref{sec:oracle} that
$m_{p+1}=|\mathcal I_{p+1}|$, so $m_1=D$ and
$m_{p+1}=2D$ for $p\ge1$.
Define vectors $R_p^\theta(x,z),R_p^0(z)\in\R^{TD}$ coordinatewise by
\begin{equation}
 (R_p^\theta(x,z))_k
 :=\frac r{\sqrt{m_{p+1}}}S_{p,k}^\theta(x,z),
 \qquad
 (R_p^0(z))_k
 :=\frac r{\sqrt{m_{p+1}}}S_{p,k}^0(z)
 \quad(k\in\mathcal I_{p+1}),
 \label{eq:extended-random-vectors}
\end{equation}
and set both vectors equal to zero outside $\mathcal I_{p+1}$.  Finally,
\begin{align}
 \bar g_p^\theta(x,z)&:=\bar B_p^\theta(x)+R_p^\theta(x,z),
 &
 \bar g_{p,0}^{\theta_{\le p}}(x,z)
 &:=\bar B_p^\theta(x)+R_p^0(z),
 &&0\le p<T,
 \label{eq:extended-response-maps}\\
 \bar g_T^\theta(x,z)&:=\nabla F_\theta(x).
 \nonumber
\end{align}
For every Borel set $\mathcal A\subseteq\R^{TD}$, extend the notation from
Section~\ref{sec:oracle} by
\begin{align}
 \mathsf P_p^\theta(x,\mathcal A)
 &:=\Prob_Z\bigl(\bar g_p^\theta(x,Z)\in\mathcal A\bigr),
 &&0\le p\le T,
 \nonumber\\
 \mathsf Q_p^{\theta_{\le p}}(x,\mathcal A)
 &:=\Prob_Z\bigl(
     \bar g_{p,0}^{\theta_{\le p}}(x,Z)\in\mathcal A\bigr),
 &&0\le p<T.
 \label{eq:extended-laws}
\end{align}
\end{definition}

The definitions above separate the part of a response that is already
determined by $\theta_{\le p}$ from the part that may depend on the next
sign vector $\theta_{p+1}$.  The vector $q_\theta^{[p]}(x)$ sets all
encoded coordinates after $p+1$ to zero, and $\Gamma_p^\theta(x)$ is the
resulting full-dimensional vector obtained from the corresponding chain
derivatives.  The vector $\bar B_p^\theta(x)$ retains only the terms
determined by $\theta_{\le p}$, while
\[
  H_p^\theta(x)
  =\Gamma_p^\theta(x)-\bar B_p^\theta(x)
\]
collects the remaining terms, which are supported on
$\mathcal I_{p+1}$ and may depend on $\theta_{p+1}$.  Recall that
$m_{p+1}=|\mathcal I_{p+1}|$; the normalization by
$\sqrt{m_{p+1}}$ keeps the norm of every random response fixed.
The distribution $\mathsf P_p^\theta$ uses slightly biased signs whose
mean contribution is $H_p^\theta(x)$, whereas
$\mathsf Q_p^{\theta_{\le p}}$ replaces them by fair signs and is therefore
the same for all choices of $\theta_{p+1}$.  These distributions are
defined for every $x$. In Lemma~\ref{lem:extended-response}, we show that whenever $p_\theta(x)=p$,
$\mathsf P_p^\theta(x,\cdot)$ is exactly the actual oracle-response distribution, and that similar KL bounds also hold for the extension. 

\begin{lemma}[Properties of the extended response maps]
\label{lem:extended-response}
For every $0\le p<T$, the maps in
Definition~\ref{def:extended-response} have the following properties.
\begin{enumerate}[label=\textup{(\alph*)}]
\item $\bar B_p^\theta(x)$ depends on $\theta$ only through
      $\theta_{\le p}$, while $H_p^\theta(x)$ depends on $\theta$ only
      through $\theta_{\le p+1}$.  Moreover, $H_0^\theta$ is supported on
      block $1$, $H_p^\theta$ is supported on blocks $p,p+1$ when $p\ge1$,
      and
      \begin{equation}
       |(H_p^\theta(x))_k|\le\frac{H\eps}{\sqrt D},
       \qquad
       \|H_p^\theta(x)\|\le\sqrt2H\eps.
       \label{eq:extended-H-bounds}
      \end{equation}
\item The response maps are jointly Borel measurable.  Their expectations
      satisfy
      $\E_Z\bar g_p^\theta(x,Z)=\Gamma_p^\theta(x)$.  Moreover,
      $\mathsf Q_p^{\theta_{\le p}}(x,\cdot)$ is the same for any two
      sign sequences having the same prefix $\theta_{\le p}$.
\item For every $x\in\R^{TD}$,
      \begin{equation}
       \KL\!\left(
       \mathsf P_p^\theta(x,\cdot)
       \,\middle\|\,
       \mathsf Q_p^{\theta_{\le p}}(x,\cdot)
       \right)
       \le KD\frac{\eps^2}{\sigma^2}.
       \label{eq:extended-kl}
      \end{equation}
\item If $p_\theta(x)=p$, then
      \begin{equation}
       \Gamma_p^\theta(x)=\nabla F_\theta(x),
       \qquad
       \bar B_p^\theta(x)=B_{p+1}(x),
       \qquad
       H_p^\theta(x)=h_{p+1}(x),
       \label{eq:response-agreement}
      \end{equation}
      and $\bar g_p^\theta(x,z)=g_\theta(x,z)$ for every $z\in\calZ$.
\end{enumerate}
\end{lemma}

\begin{proof}
The local dependence in Lemma~\ref{lem:chain} shows that
$b_i(q_\theta^{[p]})$ depends only on the entries with indices
$j\in[T]$ satisfying $|j-i|\le1$.  Consequently every term in
$\bar B_p^\theta$ uses only $\theta_{\le p}$; in particular,
$b_p(q_\theta^{[p-1]})$ contains no $q_{p+1}^\theta$.  Directly from
\eqref{eq:extended-decomposition}, we have
\[
 H_0^\theta=\mathbf C_1^{[0],\theta},
 \qquad
 H_p^\theta=
 \mathbf C_p^{[p],\theta}
 -\widetilde{\mathbf C}_p^{[p],\theta}
 +\mathbf C_{p+1}^{[p],\theta}
 \quad(p\ge1).
\]
This proves the dependence and support assertions.  Since $|b_i|\le G$,
$|b_p(q_\theta^{[p]})-b_p(q_\theta^{[p-1]})|\le2G$, and
$A/\delta=4\eps$, the coordinate estimates used in
\eqref{eq:h-bounds} apply without change and give
\eqref{eq:extended-H-bounds}.

The accuracy assumption \eqref{eq:noise-regime}, together with
$m_{p+1}\le2D$ and $r=\sigma/2$, implies
\[
 |\beta_{p,k}^\theta(x)|
 \le\frac{\sqrt{2D}}r\frac{H\eps}{\sqrt D}
 =\frac{2\sqrt2H\eps}{\sigma}\le\frac12.
\]
Thus all thresholds in \eqref{eq:extended-signs} lie in $[0,1]$.  The
same calculation as in Section~\ref{sec:oracle} gives
$\E_ZR_p^\theta(x,Z)=H_p^\theta(x)$ and hence
$\E_Z\bar g_p^\theta(x,Z)=\Gamma_p^\theta(x)$.  Equations
\eqref{eq:extended-C}--\eqref{eq:extended-response-maps} are finite
compositions and finite case distinctions of Borel functions, so the maps
are jointly Borel measurable.  The formula for $\bar B_p^\theta$ and the
fair signs in $R_p^0$ also show directly that the reference distribution is
the same for any two sign sequences having the same prefix
$\theta_{\le p}$.

The biased and fair responses have the same finite support.  Independence
of their active signs and the scalar inequality used in
\eqref{eq:one-response-kl} give
\begin{align*}
 \KL\!\left(
 \mathsf P_p^\theta(x,\cdot)
 \,\middle\|\,
 \mathsf Q_p^{\theta_{\le p}}(x,\cdot)
 \right)
 &\le\sum_{k\in\mathcal I_{p+1}}
        (\beta_{p,k}^\theta(x))^2\\
 &=\frac{m_{p+1}}{r^2}\|H_p^\theta(x)\|^2
 \le16H^2D\frac{\eps^2}{\sigma^2}
 =KD\frac{\eps^2}{\sigma^2},
\end{align*}
which proves \eqref{eq:extended-kl}.

It remains to verify agreement with the actual oracle.  Suppose
$p_\theta(x)=p$.  Then $|q_i^\theta(x)|\le1/4$ for every $i>p$, and
$q_\theta^{[p]}(x)$ agrees with the original encoded vector through
coordinate $p+1$.  Nearest-neighbor dependence gives
$b_i(q_\theta^{[p]}(x))=b_i(q^\theta(x))$ for $i\le p$.  For $i=p+1$,
the only possible term involving $q_{p+2}^\theta$ is multiplied by
$\Psi'(\pm2q_{p+1}^\theta)$, which vanishes because
$|q_{p+1}^\theta|\le1/4$.  Hence the equality also holds for $i=p+1$.
Lemma~\ref{lem:chain} gives $C_i=0$ for $i\ge p+2$.  Comparing
\eqref{eq:extended-decomposition} with \eqref{eq:Bj}--\eqref{eq:hj}
now proves \eqref{eq:response-agreement}.  The biased signs and their use of
the sample coordinates $z_k$ are then identical in
\eqref{eq:extended-response-maps} and \eqref{eq:bounded-response}, proving
$\bar g_p^\theta(x,z)=g_\theta(x,z)$ for every $z$.
\end{proof}

\begin{remark}
The KL bound above is the point where the proof differs from a
rare-reveal-style construction under the bounded variance assumption.  Here, every response remains
uniformly bounded, and the difficulty is to control the sum of many small
conditional KL divergences rather than the probability of one unboundedly
large reveal.
\end{remark}

\subsection{A comparison interaction}

We next define a sequence that uses the same algorithm maps $X_t$, private
seed $U$, sign vectors $\Theta$, and samples $Z_t$ as the actual oracle
interaction.  Its additional nondecreasing integer $\nu_t$ restricts the
next index $\rho_t$ to be at most $\nu_{t-1}+1$.

\begin{definition}[Comparison interaction]
\label{def:comparison-interaction}
Set $\nu_0=0$.  For $1\le t\le N$, define recursively
\begin{align}
 x_t^\circ
 &:=X_t(U,Y_1^\circ,\ldots,Y_{t-1}^\circ),
 \nonumber\\
 \rho_t
 &:=\max\!\left(
 \left\{i\in[T]:i\le\nu_{t-1}+1,
       |q_i^\Theta(x_t^\circ)|>\frac14\right\}\cup\{0\}
 \right),
 \nonumber\\
 Y_t^\circ
 &:=\bar g_{\rho_t}^\Theta(x_t^\circ,Z_t),
 \qquad
 \nu_t:=\max\{\nu_{t-1},\rho_t\}.
 \label{eq:comparison-recursion}
\end{align}
After the $N$th response, define
\begin{align}
 x_{N+1}^\circ
 &:=X_{N+1}(U,Y_1^\circ,\ldots,Y_N^\circ),
 \nonumber\\
 \rho_{N+1}
 &:=\max\!\left(
 \left\{i\in[T]:i\le\nu_N+1,
       |q_i^\Theta(x_{N+1}^\circ)|>\frac14\right\}\cup\{0\}
 \right),
 \qquad
 \nu_{N+1}:=\max\{\nu_N,\rho_{N+1}\}.
 \label{eq:comparison-output}
\end{align}
No response is generated at $x_{N+1}^\circ$.  Finally, define the event
\begin{equation}
 \mathcal J:=
 \left\{\exists\,s\in[N+1],\ \exists\,i\in[T]:
 i>\nu_{s-1}+1\ 
 \text{ and }\ |q_i^\Theta(x_s^\circ)|>\frac14\right\}.
 \label{eq:skipped-index-event}
\end{equation}
\end{definition}

In other words, for $t\le N$,
$x_t^\circ$ is the algorithm's $t$-th query, $\rho_t$ is the index of the
response distribution used there, and $Y_t^\circ$ is the response.  The
value $\rho_{N+1}$ is used only to update $\nu_{N+1}$ at the final output;
no response is generated there.  The next lemma explains why this
comparison is useful.

\begin{lemma}[Comparison with the actual oracle interaction]
\label{lem:comparison-interaction}
For the recursion in Definition~\ref{def:comparison-interaction}:
\begin{enumerate}[label=\textup{(\alph*)}]
\item for $s\in[N+1]$ and $k\in\{0,\ldots,T\}$, let
      \begin{equation}
       \mathcal H_{s,k}:=\sigma\!\left(
       U,Z_1,\ldots,Z_{s-1},
       \Theta_1,\ldots,\Theta_{\min\{k+1,T\}}
       \right),
       \label{eq:comparison-sigma-field}
      \end{equation}
      with empty lists omitted.  Then $\{\nu_{s-1}=k\}$ and
      $x_s^\circ\ind\{\nu_{s-1}=k\}$ are
      $\mathcal H_{s,k}$-measurable.  Every $\Theta_i$ with $i>k+1$ is
      independent of $\mathcal H_{s,k}$ and is uniform on
      $\{-1,1\}^D$;
\item
      \begin{equation}
       \Prob(\mathcal J)\le\frac1{128};
       \label{eq:jump-bound}
      \end{equation}
\item if the actual interaction from Section~\ref{sec:model} is run with
      objective $F_\Theta$, oracle $g_\Theta$, and the same
      $(U,Z_1,\ldots,Z_N)$, then on $\mathcal J^c$,
      \begin{equation}
       x_t=x_t^\circ,\quad Y_t=Y_t^\circ\quad(1\le t\le N),
       \qquad \widehat x_N=x_{N+1}^\circ.
       \label{eq:actual-comparison-agreement}
      \end{equation}
      Moreover, on $\mathcal J^c$,
      \begin{equation}
       p_\Theta(\widehat x_N)=T
       \quad\Longrightarrow\quad \nu_{N+1}=T.
       \label{eq:output-progress-implication}
      \end{equation}
\end{enumerate}
\end{lemma}

\begin{proof}
We prove (a) by induction on $s$.  The first query $x_1^\circ=X_1(U)$
uses no sign vector.  Suppose the assertion holds before query $s$.  On
$\{\nu_{s-1}=k\}$, the definition of $\rho_s$ examines only indices at
most $k+1$.  If the updated value is $\nu_s=k'$, then
$\rho_s\le k'$ and the response map $\bar g_{\rho_s}^\Theta$ depends on
sign vectors only through $\Theta_{\le\min\{\rho_s+1,T\}}\subseteq
\Theta_{\le\min\{k'+1,T\}}$.  Consequently $Y_s^\circ$, and hence the
next query selected from $U,Y_1^\circ,\ldots,Y_s^\circ$, has the asserted
measurability.  The event specifying the new value $k'$ is determined by
the same variables.  This proves the two measurability assertions by
induction.  Independence follows because $U$, the samples, and all sign
vectors are mutually independent.

For (b), fix $s,i,k$ with $i>k+1$.  On
$\{\nu_{s-1}=k\}$, part (a) makes $a_i(x_s^\circ)$ measurable with
respect to $\mathcal H_{s,k}$, while $\Theta_i$ is independent of that
sigma-field.  Lemma~\ref{lem:sign-correlation} therefore gives
\begin{equation}
 \Prob\!\left(\nu_{s-1}=k,
 |q_i^\Theta(x_s^\circ)|>\frac14\right)
 \le2e^{-D/128}\Prob(\nu_{s-1}=k).
 \label{eq:conditional-sign-bound}
\end{equation}
Summing \eqref{eq:conditional-sign-bound} over $k=0,\ldots,i-2$ bounds
the probability for each fixed pair $(s,i)$ by $2e^{-D/128}$.
There are at most $(N+1)T$ pairs $(s,i)$.  Hence
\[
 \Prob(\mathcal J)
 \le2(N+1)T e^{-D/128}
 \le\frac1{128},
\]
where the last inequality follows from \eqref{eq:n0D}.

For (c), use induction on the query number.  Suppose that the two sequences
agree before a query and that $\mathcal J$ has not occurred there.  Then no
index larger than $\nu_{t-1}+1$ satisfies the inequality in
\eqref{eq:random-last-crossed}, so
$p_\Theta(x_t^\circ)=\rho_t$.  Lemma~\ref{lem:extended-response}(d) gives
the equality of the two response maps when $\rho_t<T$.  When
$\rho_t=T$, both maps equal $\nabla F_\Theta(x_t^\circ)$ by definition.
Thus
$g_\Theta(x_t^\circ,Z_t)=
\bar g_{\rho_t}^\Theta(x_t^\circ,Z_t)$,
and the two responses, as well as the next queries selected by the same
map $X_{t+1}$, agree.  The same argument at the final output proves
\eqref{eq:actual-comparison-agreement}.  If
$p_\Theta(\widehat x_N)=T$ on $\mathcal J^c$, then
$q_T^\Theta(x_{N+1}^\circ)$ satisfies the defining inequality for
$\rho_{N+1}$; thus $\rho_{N+1}=T$ and \eqref{eq:output-progress-implication}
follows.
\end{proof}

\subsection{The response cost for one coordinate}

We now fix $i$ and analyze exactly the responses that can depend on
$\Theta_i$.  The shift by one is important: by
Definition~\ref{def:extended-response}, the relevant response distribution
is $\mathsf P_{i-1}^\Theta$, not $\mathsf P_i^\Theta$.

\begin{definition}[Counts for coordinate $i$]
\label{def:coordinate-counts}
For $i\in[T]$ and $s\in[N+1]$, define
\begin{equation}
 R_i(s):=\sum_{t=1}^{s-1}
 \ind\{\nu_{t-1}<i,\ \rho_t=i-1\}.
 \label{eq:responses-before-query}
\end{equation}
Thus $R_i(s)$ is the number of responses already produced before
$x_s^\circ$ was selected that had response index $i-1$ while
$\nu_{t-1}<i$.  For an integer $n\ge1$, define
\begin{equation}
 \mathcal A_{i,n}:=
 \left\{\exists\,s\in[N+1]:
 \nu_{s-1}=i-1,\quad R_i(s)<n,\quad
 |q_i^\Theta(x_s^\circ)|>\frac14\right\}.
 \label{eq:early-coordinate-event}
\end{equation}
Also define the total count
\begin{equation}
 L_i:=\sum_{t=1}^N
 \ind\{\nu_{t-1}<i,\ \rho_t=i-1\}.
 \label{eq:Li-definition}
\end{equation}
\end{definition}

\begin{lemma}[Probability of reaching one coordinate with few responses]
\label{lem:one-link}
For every $i\in[T]$ and every integer $n\ge1$,
\begin{equation}
 \Prob(\mathcal A_{i,n})
 \le
 \frac{nKD\eps^2/\sigma^2+\log2}
 {D/128-\log(2(N+1))}.
 \label{eq:one-link-bound}
\end{equation}
In particular,
\begin{equation}
 \Prob(\mathcal A_{i,n_0})
 \le\frac3{128}<\frac1{32}.
 \label{eq:quick-crossing}
\end{equation}
\end{lemma}

\begin{proof}
We condition only on variables that are available before the first response
whose distribution can depend on $\Theta_i$.  For $i=1$, this means
conditioning on $U$ before $x_1^\circ=X_1(U)$ is examined.  For $i\ge2$,
let
\begin{equation}
 \tau_i:=\min\{t\in[N]:\nu_{t-1}=i-2,\ \rho_t=i-1\},
 \label{eq:entry-time}
\end{equation}
with the minimum equal to $\infty$ if the set is empty.  If
$\mathcal A_{i,n}$ occurs, then $\tau_i\le N$.  Fix $t\in[N]$ and condition
on a value of
\begin{equation}
 \bigl(\Theta_{<i},U,Y_1^\circ,\ldots,Y_{t-1}^\circ,
 x_t^\circ,\rho_t,\nu_0,\ldots,\nu_{t-1}\bigr)
 \label{eq:entry-data}
\end{equation}
for which $\tau_i=t$.  Before time $t$, every response index is at most
$i-2$; its response map therefore depends only on
$\Theta_{\le i-1}$.  The conditions $\nu_{t-1}=i-2$ and $\rho_t=i-1$ are
also determined by the variables in \eqref{eq:entry-data} without using
$\Theta_i$.  It follows that, under this conditional law, $\Theta_i$ is
still uniform and independent of the listed variables.  This conclusion is
also true for $i=1$ after conditioning on $U$.

Denote the conditioned value just fixed by $\eta$, and let $\Prob_\eta$ denote
the corresponding regular conditional probability.  For $i=1$, $\eta$ is the
fixed value of $U$, and we put $t=0$.  We now define new random variables
for a finite continuation; they are denoted by
$x_s^{(i)},Y_s^{(i)},\lambda_s^{(i)}$, and $c_s^{(i)}$ and are distinct
from the comparison variables in
Definition~\ref{def:comparison-interaction}.  Set
\begin{equation}
 s_i:=
 \begin{cases}
  1,&i=1,\\
  t+1,&i\ge2.
 \end{cases}
 \label{eq:first-coordinate-query}
\end{equation}
For $i\ge2$, first generate
\begin{equation}
 Y_t^{(i)}\sim
 \mathsf P_{i-1}^\Theta(x_t^\circ,\cdot)
 \quad\text{and set}\quad c_{t+1}^{(i)}:=1.
 \label{eq:entry-response}
\end{equation}
Thus $x_t^\circ$ is not examined again for coordinate $i$: its response is
the first response counted for that coordinate, and the first query examined
for coordinate $i$ is $x_{t+1}^{(i)}$ (which is the final output when
$t=N$).  For $i=1$, set $c_1^{(1)}:=0$, so $x_1^{(1)}=X_1(U)$ is examined
before any response is counted.

For every $s_i\le s\le N+1$, define the query recursively by
\begin{equation}
 x_s^{(i)}:=
 \begin{cases}
 X_s(U,Y_1^{(1)},\ldots,Y_{s-1}^{(1)}),&i=1,\\
 X_s(U,Y_1^\circ,\ldots,Y_{t-1}^\circ,
       Y_t^{(i)},\ldots,Y_{s-1}^{(i)}),&i\ge2.
 \end{cases}
 \label{eq:coordinate-continuation-query}
\end{equation}
with empty lists omitted.
For $s\le N$, set
\begin{equation}
 \lambda_s^{(i)}:=
 \max\!\left(
 \left\{j\in\{1,\ldots,i-1\}:
 |q_j^\Theta(x_s^{(i)})|>\frac14\right\}\cup\{0\}
 \right),
 \label{eq:fixed-continuation-index}
\end{equation}
where $\{1,\ldots,i-1\}$ is empty when $i=1$.  Conditional on all
previously defined variables and on $\Theta_i$, generate
\begin{equation}
 Y_s^{(i)}\sim
 \begin{cases}
 \mathsf P_{\lambda_s^{(i)}}^\Theta(x_s^{(i)},\cdot),
       &c_s^{(i)}<n,\\
 \text{the point mass at }0\in\R^{TD},&c_s^{(i)}=n,
 \end{cases}
 \label{eq:coordinate-continuation-response}
\end{equation}
and update
\begin{equation}
 c_{s+1}^{(i)}
 :=c_s^{(i)}+
 \ind\{c_s^{(i)}<n,\ \lambda_s^{(i)}=i-1\}.
 \label{eq:coordinate-continuation-count}
\end{equation}
The indicator of $|q_i^\Theta(x_s^{(i)})|>1/4$ is absent from
\eqref{eq:fixed-continuation-index}--\eqref{eq:coordinate-continuation-count}
and therefore does not affect any later query or response.  Define the event
on this continuation by
\begin{equation}
 E_{i,n}^{(t)}:=
 \left\{\exists\,s\in\{s_i,\ldots,N+1\}:
 c_s^{(i)}<n,\quad
 |q_i^\Theta(x_s^{(i)})|>\frac14\right\}.
 \label{eq:continuation-event}
\end{equation}
This is a fixed finite construction for every value of $\Theta_i$; queries
chosen after $c_s^{(i)}=n$ do not enter $E_{i,n}^{(t)}$.

For the conditioned value $\eta$, let $\mathbf P_{i,\eta}$ be the joint law of
$\Theta_i$, the queries in \eqref{eq:coordinate-continuation-query}, and
all newly generated responses in
\eqref{eq:entry-response} and
\eqref{eq:coordinate-continuation-response}.  Let $\mathbf Q_{i,\eta}$ use
the same recursion, except that the response in \eqref{eq:entry-response}
when $i\ge2$, and every later response with
$c_s^{(i)}<n$ and $\lambda_s^{(i)}=i-1$, is sampled from
$\mathsf Q_{i-1}^{\Theta_{<i}}$ instead of
$\mathsf P_{i-1}^\Theta$.  At most $n$ response distributions are replaced.

Until the earlier of the first query satisfying
$|q_i^\Theta(x_s^{(i)})|>1/4$ and the production of the $n$th response
with index $i-1$, induction shows that this continuation and
\eqref{eq:comparison-recursion} have the same queries and response
distributions.  Indeed, before either event, the comparison recursion has
$\nu_{s-1}=i-1$, the displayed inequality is false, and its response index
equals $\lambda_s^{(i)}$.  Once the $n$th such response has been produced,
$R_i(s)\ge n$ at every later query, so $\mathcal A_{i,n}$ cannot occur.
Consequently
\begin{equation}
 \Prob_\eta(\mathcal A_{i,n})
 =\mathbf P_{i,\eta}(E_{i,n}^{(t)}).
 \label{eq:continuation-agreement}
\end{equation}

Under $\mathbf Q_{i,\eta}$, a response with index at most $i-2$ depends only
on $\Theta_{<i}$, every response with index $i-1$ uses the reference
distribution, and each remaining response is the fixed vector $0$.
Induction over $s$ therefore shows that every $x_s^{(i)}$ is independent of
$\Theta_i$.  Lemma~\ref{lem:sign-correlation} and a union bound over at
most $N+1$ query positions give
\begin{equation}
 \mathbf Q_{i,\eta}(E_{i,n}^{(t)})
 \le2(N+1)e^{-D/128}.
 \label{eq:reference-event-bound}
\end{equation}
The union bound may include positions with $c_s^{(i)}=n$, although those
positions are excluded from $E_{i,n}^{(t)}$.

The conditional KL chain rule and Lemma~\ref{lem:extended-response}(c) give
\begin{equation}
 \KL(\mathbf P_{i,\eta}\|\mathbf Q_{i,\eta})
 \le nKD\frac{\eps^2}{\sigma^2}.
 \label{eq:coordinate-kl-chain}
\end{equation}
Indeed, the two laws use identical conditional distributions at all but at
most $n$ response positions with index $i-1$, and each such position
contributes at most the right side of \eqref{eq:extended-kl}.

For $u,v\in[0,1]$, define the binary divergence
\[
 \operatorname{kl}(u\|v)
 :=u\log\frac uv+(1-u)\log\frac{1-u}{1-v},
\]
with the usual continuous conventions at the endpoints.  If
$u=\mathbf P_{i,\eta}(E_{i,n}^{(t)})$ and
$v=\mathbf Q_{i,\eta}(E_{i,n}^{(t)})$, data processing,
\eqref{eq:reference-event-bound}, and \eqref{eq:coordinate-kl-chain} give
\[
 nKD\frac{\eps^2}{\sigma^2}
 \ge\operatorname{kl}(u\|v)
 \ge u\log(1/v)-\log2
 \ge u\bigl(D/128-\log(2(N+1))\bigr)-\log2.
\]
The denominator is positive by \eqref{eq:n0D}, so rearranging proves
\eqref{eq:one-link-bound} almost every conditioning value \(\eta\). Integrating with
respect to the conditioning law proves the unconditional bound.

Finally, \eqref{eq:n0D} implies
\[
 n_0KD\eps^2/\sigma^2\le D/2^{14},
 \qquad
 \log2\le D/2^{15},
 \qquad
 D/128-\log(2(N+1))\ge D/256.
\]
Substitution into \eqref{eq:one-link-bound} proves
\eqref{eq:quick-crossing}.
\end{proof}

\subsection{Summing the response counts}

Combining all preceding results, we arrive at the proposition below.

\begin{proposition}[Information bound for the algorithm's output]
\label{prop:information-bound}
If $N\le Tn_0/2$, then
\begin{equation}
 \Prob\!\left(\|\nabla F_\Theta(\widehat x_N)\|>2\eps\right)
 \ge\frac{119}{128}.
 \label{eq:gradient-probability}
\end{equation}
\end{proposition}

\begin{proof}
A response at time $t$ can be included in $L_i$ only for
$i=\rho_t+1$.  Therefore the counts in \eqref{eq:Li-definition} satisfy
\begin{equation}
 \sum_{i=1}^T L_i\le N.
 \label{eq:Li-sum}
\end{equation}
Define
\[
 S_i:=\ind\{\nu_{N+1}\ge i,\ L_i<n_0\}.
\]
If $\nu_{N+1}\ge i$, let $s$ be the first query or the final output for
which $\nu_s\ge i$.  The update rule implies
$\nu_{s-1}=i-1$, $|q_i^\Theta(x_s^\circ)|>1/4$, and
$R_i(s)=L_i$, because no term in $L_i$ is counted after $\nu$ reaches $i$.
Consequently
\[
 \{\nu_{N+1}\ge i,\ L_i<n_0\}=\mathcal A_{i,n_0},
\]
and Lemma~\ref{lem:one-link} gives
\begin{equation}
 \E S_i\le\frac1{32},
 \qquad
 \E\sum_{i=1}^T S_i\le\frac T{32}.
 \label{eq:short-count}
\end{equation}

On $\{\nu_{N+1}=T\}$ every $i\in[T]$ satisfies
$\nu_{N+1}\ge i$.  Hence, using \eqref{eq:Li-sum} and the assumed budget,
\[
 \sum_{i=1}^T S_i
 =T-\sum_{i=1}^T\ind\{L_i\ge n_0\}
 \ge T-\frac1{n_0}\sum_{i=1}^TL_i
 \ge\frac T2.
\]
Markov's inequality and \eqref{eq:short-count} therefore yield
\begin{equation}
 \Prob(\nu_{N+1}=T)
 \le\Prob\!\left(\sum_{i=1}^TS_i\ge\frac T2\right)
 \le\frac1{16}.
 \label{eq:comparison-completion}
\end{equation}

By Lemma~\ref{lem:comparison-interaction}, on $\mathcal J^c$ the actual
output equals $x_{N+1}^\circ$, and
$p_\Theta(\widehat x_N)=T$ implies $\nu_{N+1}=T$.  Thus
\[
 \{p_\Theta(\widehat x_N)=T\}
 \subseteq\mathcal J\cup\{\nu_{N+1}=T\},
\]
so \eqref{eq:jump-bound} and \eqref{eq:comparison-completion} give
\begin{equation}
 \Prob\bigl(p_\Theta(\widehat x_N)=T\bigr)
 \le\frac1{128}+\frac1{16}=\frac9{128}.
 \label{eq:actual-completion}
\end{equation}
Finally, Lemma~\ref{lem:certificate} gives
\[
 \|\nabla F_\Theta(\widehat x_N)\|\le2\eps
 \quad\Longrightarrow\quad
 |q_i^\Theta(\widehat x_N)|>\frac12\quad\text{for every }i\in[T],
\]
which implies $p_\Theta(\widehat x_N)=T$.  Taking complements in
\eqref{eq:actual-completion} proves \eqref{eq:gradient-probability}.
\end{proof}

This completes the core part of the proof.

\section{Completing the lower-bound construction}
\label{sec:parameters}

Recall that $T$ is the number of links and $n_0$
is the number of responses needed per link.  The small-accuracy condition
ensures that the floors in their definitions lose at most a factor of two.

Let
\begin{equation}
 c_0=2^{-20},
 \qquad
 c=\frac1{2560C\,2^{39}},
 \qquad
 Q=\frac{\Delta L}{\eps^2},
 \qquad
 m=\frac{\sigma^2}{\eps^2}.
 \label{eq:constants-again}
\end{equation}
Under \eqref{eq:accuracy-regime}, both $Q$ and $m$ are at least $2^{40}$.
In particular, $Q/(320C)>2$ and
$m/(2^{14}K)=m/2^{38}\ge4$.  Since $\lfloor x\rfloor\ge x/2$ for
$x\ge2$, the floors in \eqref{eq:ATdelta} and \eqref{eq:n0D} satisfy
\begin{equation}
 T\ge\frac{Q}{640C},
 \qquad
 n_0\ge\frac{m}{2^{39}},
 \qquad
 \frac{Tn_0}{2}\ge\frac{Qm}{1280C\,2^{39}}.
 \label{eq:floor-bounds}
\end{equation}
Since $m\ge1$, $Q+Qm\le2Qm$.  Thus
\begin{equation}
 N\le c(Q+Qm)\quad\Longrightarrow\quad N\le\frac{Tn_0}{2}.
 \label{eq:budget-to-links}
\end{equation}
Combining \eqref{eq:gradient-probability} and
\eqref{eq:budget-to-links},
\begin{equation}
 \E_{\Theta,U,Z_{1:N}}\|\nabla F_\Theta(\widehat x_N)\|
 >2\eps\frac{119}{128}>\eps.
 \label{eq:average-failure}
\end{equation}

The prior on $\Theta$ is uniform on the finite set $\{-1,1\}^{TD}$.  Hence
\eqref{eq:average-failure} is an average of finitely many conditional
expectations over $U,Z_{1:N}$.  There exists a deterministic
$\theta^\star$ for which
\[
 \E_{U,Z_{1:N}}\|\nabla F_{\theta^\star}(\widehat x_N)\|>\eps.
\]
We choose $\theta^\star$ after the algorithm is fixed, but before any
realization of $U,Z_{1:N}$ is drawn, so that the objective, oracle map, and product-uniform sample law are then fixed throughout the interaction.

The dimension, chosen before the algorithm on $\R^d$, is
\begin{equation}
 d=TD
 =T\left\lceil\max\left\{2^{15},
 256\log\bigl(16(N+1)T\bigr)\right\}\right\rceil.
 \label{eq:dimension}
\end{equation}
It depends only on the displayed parameters and $N$.  Eliminating $T$ from
the right side gives
\begin{equation}
 d\le\frac{Q}{320C}
 \left[2^{15}+1+256\log\left(
 \frac{16(N+1)Q}{320C}\right)\right].
 \label{eq:dimension-bound}
\end{equation}
Indeed, with $X=Q/(320C)$, one has $1\le T\le X$ and
\[
 D\le2^{15}+1+256\log(16(N+1)T)
 \le2^{15}+1+256\log(16(N+1)X).
\]
Multiplying by $T\le X$ proves \eqref{eq:dimension-bound}.
\\

Lemmas~\ref{lem:objective} and \ref{lem:oracle} verify every property of the
model.  Equations \eqref{eq:constants-again}--\eqref{eq:average-failure}
prove Theorem~\ref{thm:main}.  Because $m\ge1$, the product term dominates
the deterministic term in the stipulated regime, which permits the stated
sum.  \hfill$\square$

\section{Discussion and limitations}
\label{sec:discussion}

Theorem~\ref{thm:main} shows that 
requiring the stochastic-gradient error to be uniformly bounded almost
surely does not improve the worst-case rate obtained under bounded
variance.  In the regime of the theorem, every randomized adaptive
algorithm requires
\[
  \Omega\!\left(
    \frac{\Delta L}{\eps^2}
    +
    \frac{\Delta L\sigma^2}{\eps^4}
  \right)
\]
queries.  The hard oracle is unbiased, uses one fixed sample law, and
satisfies the noise bound on a common probability-one set for all query
points.

The main novelty is the replacement of the rare-revelation mechanism used
in bounded-variance lower bounds.  There, the next chain coordinate is
revealed with a small probability \(\rho\), and unbiasedness requires a
response of size proportional to \(1/\rho\).  This gives the desired delay
but violates a uniform noise bound.  Here the next coordinate is encoded by
a hidden sign vector, and its gradient contribution is represented as the
mean of bounded, slightly biased signs.  Every response remains uniformly
bounded, while its distribution contains only
\(O(D\eps^2/\sigma^2)\) KL information about the hidden sign vector.

This change also requires a different argument for adaptive algorithms.
The bounded-variance proof directly counts Bernoulli reveal events.  In the
present construction, every response may contain a small amount of
information, so the proof instead accumulates conditional KL divergences.
The extension of the one-response laws to every pair \((p,x)\) makes the KL
chain rule applicable to arbitrary queries generated by random histories,
and the fixed-length comparison removes the dependence of the comparison
horizon on the hidden sign vector.  These steps show that one link requires
\(\Omega(\sigma^2/\eps^2)\) relevant responses.  A separate
\(\log\cosh\) term prevents large-norm queries from exploiting saturation
of the \(\tanh\) encoding.

The result applies only to the oracle class in
Definition~\ref{def:oracle}.  The constructed map \(g(\cdot,z)\) is
generally discontinuous and is not shown to be the gradient of a sample
function.  Thus continuous, samplewise-smooth, conservative, and finite-sum
oracles require separate constructions.  The proof also assumes \(K=1\);
when one sample can be evaluated at several query points, the resulting
cross-query correlations are not controlled by the present argument.

\end{document}